\documentclass{amsart}

\usepackage{enumerate}
\usepackage{amsmath,amsfonts,amssymb,amsthm}
\usepackage{csquotes}
\usepackage{graphicx,amsaddr}

\newtheorem{thm}{Theorem}[section]
\newtheorem{lem}[thm]{Lemma}
\newtheorem{prop}[thm]{Proposition}
\newtheorem{cor}[thm]{Corollary}

\theoremstyle{definition}
\newtheorem{defn}[thm]{Definition}
\newtheorem{exam}[thm]{Example}

\theoremstyle{remark}

\newtheorem*{sol}{Solution}

\newcommand{\bZp}{\mathbb{Z}_+}
\newcommand{\bQp}{\mathbb{Q}_+}
\newcommand{\bRp}{\mathbb{R}_+}

\newcommand{\bN}{\mathbb{N}}

\numberwithin{equation}{section}

\title{The Operator norm on Weighted Discrete Semigroup Algebras $\ell^1(S, \omega)$}
\author{H. V. Dedania}

\address{Dept. of Mathematics, Sardar Patel University, Vallabh Vidyanagar 388120, Gujarat, India}
\email{hvdedania@gmail.com}

\author{J. G. Patel*}

\address{Dept. of Mathematics, Sardar Patel University, Vallabh Vidyanagar 388120, Gujarat, India}
\email{jatinpatel@spuvvn.edu}

\begin{document}



\begin{abstract}
Let $\omega$ be a weight on a right cancellative semigroup $S$. Let $\|\cdot\|_{\omega}$ be the weighted norm on the weighted discrete semigroup algebra $\ell^1(S, \omega)$. In this paper, we prove that the weight $\omega$ satisfies F-property if and only if the operator norm $\| \cdot \|_{\omega op}$ of $\| \cdot \|_{\omega}$ is exactly equal to another weighted norm $\| \cdot \|_{\widetilde{\omega}_1}$ [Theorem \ref{17}($\ref{9}$)]. Though its proof is elementary, the result is unexpectedly surprising. In particular, $\| \cdot \|_{1 op}$ is same as $\| \cdot \|_1$ on $\ell^1(S)$. Moreover, various examples are discussed to understand the relating among $\| \cdot \|_{\omega op}$, $\| \cdot \|_{\omega}$, and $\ell^1(S, \omega)$.
\end{abstract}
\maketitle

\section{Introduction}
Let $( A , \| \cdot \|)$ be an associative, complex Banach algebra which is \emph{faithful}, i.e. if $a \in  A $ and $a x = 0 \; (x \in  A )$, then $a = 0$. Define the operator norm $\| \cdot \|_{op}$ on $ A $ as $$ \| a \|_{op} = \sup \{ \| a x\| : x \in  A , \| x \| \leq 1\} \quad (a \in  A ).$$ Then it is always true that $\| a \|_{op} \leq \| a \| \; (a \in  A )$. The norm $\| \cdot \|$ is called \emph{regular} if $\| \cdot \|_{op} = \| \cdot \|$ on $A$. There are several sufficient conditions for the regularity of the norm $\| \cdot \|$. For example, the norm $\| \cdot \|$ is regular if one of the following conditions holds : (i) $ A $ is unital; (ii) $ A $ has a bounded approximate identity with bound 1; (iii) $\| a^2 \| = \| a \|^2 \; (a \in  A )$; (iv) $ A $ is a $\ast $-algebra and $\|a^{\ast } a \| = \| a \|^2 \; (a \in  A )$ \cite{BD:95}. However, there is no necessary and sufficient condition for the regularity of $\| \cdot \|$. It would be clear from this paper that this property depends on both $\| \cdot \|$ and $ A $.

Let $S$ be a semigroup. A \emph{weight} on $S$ is a map $\omega : S \longrightarrow (0, \infty)$ satisfying the submultiplicativity $\omega(s t) \leq \omega(s) \omega(t) \; (s, t \in S)$. Consider the Banach space $$\ell^1(S, \omega) = \{f : S \longrightarrow \mathbb{C} : \| f \|_{\omega} = \sum\limits_{s \in S} |f(s)| \omega(s) < \infty\}.$$ For $f, g \in \ell^1(S, \omega)$, the \emph{convolution product} $f \ast g$ is defined as $$f \ast g(s) = \sum \{f(u) g(v): u,v \in S, uv = s\} \quad (s \in S).$$
If $uv = s$ has no solution, then $f \ast g(s) = 0$. Then $(\ell^1(S, \omega), \| \cdot \|_{\omega}, \ast)$ is a Banach algebra; it is called the \emph{weighted discrete semigroup algebra }\cite[P.159]{Dal:00}. If $S$ is right cancellative, then it is easy to see that $\ell^1(S, \omega)$ is faithful. In the past, $\ell^1(S, \omega)$ was used to serve as counter examples. Now it has been studied systematically. Various Banach algebra properties of $\ell^1(S, \omega)$ can be characterized in terms of easy objects $S$ and $\omega$ \cite{BDD:11,DaDe:09,HeZu:56}.

\section{Results on Operator Norms}
Throughout $S$ is a right cancellative semigroup. So the Banach algebra $\ell^1(S, \omega)$ is faithful, and hence the operator seminorm $\| \cdot \|_{\omega op}$ on $(\ell^1(S, \omega), \| \cdot \|_{\omega})$ is a norm.

\begin{defn}
Let $k \in \bN$, and let $\omega = \widetilde{\omega}_0$ be a weight on $S$. Define $$\widetilde{\omega}_k(s)= \sup \{\frac{\widetilde{\omega}_{k-1}(st)}{\widetilde{\omega}_{k-1}(t)} : t \in S\} \quad (s \in S).$$
\end{defn}

\begin{prop}\label{14}
Let $k \in \bZp$ and $\omega = \widetilde{\omega}_0$ be a weight on $S$. Then
\begin{enumerate}[{(i)}]
  \item Each $\widetilde{\omega}_k$ is a weight on $S$. \label{7}
  \item $\widetilde{\omega}_{k+1}(s) \leq \widetilde{\omega}_{k}(s) \; (s \in S)$. \label{15}
  \item $\ell^1(S, \widetilde{\omega}_{k-1}) \subset \ell^1(S, \widetilde{\omega}_{k})$.
  \item $\|\cdot\|_{\widetilde{\omega}_{k}}$ is an algebra norm on $\ell^1(S, \omega)$.
  \item $\lim\limits_{n \rightarrow \infty} \widetilde{\omega}_k(s^n)^{\frac{1}{n}} = \lim\limits_{n \rightarrow \infty} \widetilde{\omega}_{k+1}(s^n)^{\frac{1}{n}} \; (s \in S)$.
\end{enumerate}
\end{prop}

\begin{proof}
  $(i)$ It is enough to prove that $\widetilde{\omega}_1$ is a weight. Let $s, t, u \in S$. Then
  $$\frac{\omega(stu)}{\omega(u)} = \frac{\omega(stu)}{\omega(tu)} \frac{\omega(tu)}{\omega(u)} \leq \widetilde{\omega}_1(s) \widetilde{\omega}_1(t).$$
 Since $u \in S$ is arbitrary, $\widetilde{\omega}_1(st) \leq \widetilde{\omega}_1(s) \widetilde{\omega}_1(t)$. Thus $\widetilde{\omega}_1$ is a weight on $S$.

  $(ii)$ Since $\widetilde{\omega}_k$ is a weight, $\frac{\widetilde{\omega}_k(su)}{\widetilde{\omega}_k(u)} \leq \widetilde{\omega}_k(s) \; (u \in S)$. Hence $\widetilde{\omega}_{k+1}(s) \leq \widetilde{\omega}_k(s)$.

  $(iii)$ This follows from $(\ref{15})$.

  $(iv)$ This is clear because $\widetilde{\omega}_k$ is a weight and $\widetilde{\omega}_k \leq \omega$ on $S$.

  $(v)$ We prove only for $k = 0$. Fix $s \in S$. Let $ \rho = \lim\limits_{n \rightarrow \infty} \omega(s^n)^{\frac{1}{n}}$ and $ \rho_1 = \lim\limits_{n \rightarrow \infty} \widetilde{\omega}_1(s^n)^{\frac{1}{n}}$. Then clearly $0 \leq \rho_1 \leq \rho$ because $\widetilde{\omega}_1 \leq \omega$ on $S$. If $\rho = 0$, then there is nothing to prove. So assume that $\rho > 0$. Define $\omega_s(n) = \omega(s^n) \; (n \in \mathbb{N})$. Clearly, $\omega_s$ is a weight on $\mathbb{N}$ and $\lim\limits_{n \rightarrow \infty} \omega_s(n)^{\frac{1}{n}} = \rho$. Let $0< \epsilon < \rho$ be arbitrary. Then there exists $n_0 \in \mathbb{N}$ such that

  \begin{eqnarray}\label{5}
  \rho - \epsilon &<& \omega_s(n)^{\frac{1}{n}} \quad (n \geq n_0).
  \end{eqnarray}

Let $\overline{\omega}_s(n) = \sup \{ \frac{\omega_s(n+k)}{\omega_s(k)} : k \in \mathbb{N} \} \; (n \in \mathbb{N})$. Then $\overline{\omega}_s$ is a weight on $\mathbb{N}$ and $\omega_s(n+1) \leq \omega_s(1) \overline{\omega}_s(n)$. By Inequality (\ref{5}), $(\rho - \epsilon)^{n+1}  < \omega_s(1) \overline{\omega}_s(n) \; (n \geq n_0)$. Thus  $\rho - \epsilon < \lim\limits_{n \rightarrow \infty} \overline{\omega}_s(n)^{\frac{1}{n}}$. Note that, for any $n \in \mathbb{N}$,

\begin{eqnarray*}
  \overline{\omega}_s(n) &=& \sup\limits_{k \in \mathbb{N}} \frac{\omega_s(n+k)}{\omega_s(k)} = \sup\limits_{k \in \mathbb{N}} \frac{\omega(s^n \cdot s^k)}{\omega(s^k)} \leq \sup\limits_{t \in S} \frac{\omega(s^n t)}{\omega(t)} = \widetilde{\omega}_1(s^n).
\end{eqnarray*}
Hence, we have $\rho - \epsilon < \lim\limits_{n \rightarrow \infty} \widetilde{\omega}_1(s^n)^{\frac{1}{n}} = \rho_1 $. This completes the proof.
\end{proof}

\begin{lem}\label{13}
Let $( A , \| \cdot \|)$ be a Banach algebra and $B$ be a dense subset of $ A $. Let $| \cdot |$ be another norm on $A$ such that $|\cdot| \leq \| \cdot \|$ on $A$. If $\| x \|_{op} = |x| \; (x \in B)$, then $\| \cdot \|_{op} = |\cdot|$ on $A$.
\end{lem}

\begin{proof}
It is easy.
\end{proof}

\begin{defn}
A weight $\omega$ on $S$ has \emph{F-property} if, for every finite subset $\{t_1, \ldots, t_n\}$ of $S$ and any $r < 1$, there exists $s \in S$ such that $\frac{\omega(t_k s)}{\omega(s)} \geq r \widetilde{\omega}_1(t_k) \; (1 \leq k \leq n).$
\end{defn}

\noindent Every constant weight has F-property. If $S$ has an identity $e_s$ and $\omega(e_s) = 1$, then $\omega$ has F-property. So far we could not find any weight which does not satisfy the F-property. Next theorem is our main result.

\begin{thm}\label{17}
Let  $k \in \bZp$ and $\omega = \widetilde{\omega}_0$ be a weight on $S$. Then
\begin{enumerate}[{(i)}]
  \item $\|\delta_t\|_{\widetilde{\omega}_{k} op} = \|\delta_t\|_{\widetilde{\omega}_{k+1}} \; (t \in S).$ \label{22}
  \item $\|f\|_{\widetilde{\omega}_k op} \leq \|f\|_{\widetilde{\omega}_{k+1}} \; (f \in \ell^1(S,\omega))$. \label{18}
  \item \label{9} $\widetilde{\omega}_k$ satisfies F-property if and only if $\|f\|_{\widetilde{\omega}_k op} = \|f\|_{\widetilde{\omega}_{k+1}} \; (f \in \ell^1(S, \omega)).$
  \item \label{23} If $\widetilde{\omega}_k$ satisfies F-property and $\widetilde{\omega}_{k+1} = \widetilde{\omega}_k$ on $S$, then $\| \cdot \|_{\widetilde{\omega}_k}$ is a regular norm.
  \item The $\ell^1$- norm $\| \cdot \|_1$ on $\ell^1(S)$ is regular.
\end{enumerate}
\end{thm}

\begin{proof}We shall prove all these statements for $k = 0$, i.e., for $\omega = \widetilde{\omega}_0$.

$(i)$ \label{16} Fix $t \in S$. Let $g = \sum\limits_{s \in S}g(s) \delta_s \in \ell^1(S,\omega)$ such that $\|g\|_{\omega} \leq 1$.
Then
\begin{eqnarray*}
  \|\delta_t \ast g\|_{\omega} &=& \|\delta_t \ast \sum\limits_{s \in S} g(s) \delta_s\|_{\omega} = \|\sum\limits_{s \in S} g(s) \delta_{ts}\|_{\omega} \leq \sum\limits_{s \in S} |g(s)| \omega(ts)\\
    &=& \sum\limits_{s \in S} |g(s)| \omega(s) \frac{\omega(ts)}{\omega(s)} \leq \sum\limits_{s \in S} |g(s)| \omega(s) \widetilde{\omega}_1(t) \leq \widetilde{\omega}_1(t).
\end{eqnarray*}
Since $g$ is arbitrary, $\|\delta_t\|_{\omega op} \leq  \widetilde{\omega}_1(t) = \| \delta_t \|_{\widetilde{\omega}_1}$.
For the reverse inequality, let $s \in S$ and $\widetilde{\delta}_s= \frac{\delta_s}{\omega(s)}$. Then $\|\widetilde{\delta}_s\|_{\omega}=1$ and
\begin{eqnarray*}
  \|\delta_t\|_{\omega op} &\geq& \|\delta_t \ast \widetilde{\delta}_s\|_{\omega} = \frac{\|\delta_{ts}\|_{\omega}}{\omega(s)} = \frac{\omega(ts)}{\omega(s)}.
\end{eqnarray*}
Since $s \in S$ is arbitrary, $\|\delta_t\|_{\omega op} \geq \widetilde{\omega}_1(t) = \| \delta_t \|_{\widetilde{\omega}_1}$. Thus $(\ref{22})$ is proved.

$(ii)$ Let $f = \sum\limits_{s \in S} f(s) \delta_{s} \in \ell^1(S,\omega)$. Let $g \in \ell^1(S,\omega)$ such that $\|g\|_{\omega} \leq 1$. Then
\begin{eqnarray*}
  \|f \ast g\|_{\omega} &=& \|(\sum\limits_{s \in S} f(s) \delta_{s}) \ast g\|_{\omega} = \|\sum\limits_{s \in S} f(s) \delta_{s} \ast g\|_{\omega} \leq \sum\limits_{s \in S} |f(s)| \|\delta_{s} \ast g\|_{\omega} \\
    &\leq& \sum\limits_{s \in S} |f(s)| \|\delta_{s}\|_{\omega op} =  \sum\limits_{s \in S} |f(s)| \|\delta_{s}\|_{\widetilde{\omega}_1} \; (\text{By Statement } (i) \textrm{ above})\\
    &=& \sum\limits_{s \in S} |f(s)| \widetilde{\omega}_1(s) = \|f\|_{\widetilde{\omega}_1}.
\end{eqnarray*}
Since $\| g \|_{\omega} \leq 1$ is arbitrary, $\| f \|_{\omega op} \leq \|f\|_{\widetilde{\omega}_1}$. This proves $(\ref{18})$.

$(iii)$ Let $c_{00}(S) = \{ f \in \ell^1(S, \omega) : \textrm{supp} f = \{s : f(s) \neq 0 \} \textrm{ is finite}\}$. Then $c_{00}(S)$ is dense in $\ell^1(S, \omega)$. So, by Lemma \ref{13} and Statement $(\ref{18})$ above, it is sufficient to prove that
$$\|f\|_{\omega op} \geq \|f\|_{\widetilde{\omega}_1} \quad (f \in c_{00}(S)).$$
Let $f \in c_{00}(S)$ and supp$f = \{t_1, \ldots, t_n\}$. Let $r <1 $ be arbitrary. Since $\omega$ satisfies the F-property, there exists $s \in S$ such that
\begin{eqnarray} \label{4}
  \frac{\omega(t_k s)}{\omega(s)} &\geq& r \widetilde{\omega}_1(s) \quad (1 \leq k \leq n).
\end{eqnarray}
Let $\widetilde{\delta}_s = \frac{\delta_s}{\omega(s)}$. Then $\|\widetilde{\delta}_s\|_{\omega} = 1$ and
\begin{eqnarray}
\nonumber  \|f\|_{\omega op} &\geq& \|f \ast \widetilde{\delta}_s\|_{\omega} = \frac{1}{\omega(s)} \|(\sum\limits_{k=1}^{n}f(t_k)\delta_{t_k}) \ast \delta_s \|_{\omega} = \frac{1}{\omega(s)} \|\sum\limits_{k=1}^{n}f(t_k) \delta_{t_k s}\|_{\omega} \\
\nonumber   &=& \frac{1}{\omega(s)} \sum\limits_{k=1}^{n}|f(t_k)| \|\delta_{t_k s}\|_{\omega} \quad (\because S \text{ is right cancellative})\\
\nonumber   &=& \sum\limits_{k=1}^{n}|f(t_k)| \frac{\omega(t_k s)}{\omega(s)} \geq \sum\limits_{k=1}^{n}|f(t_k)| r \widetilde{\omega}_1(t_k) \quad (\text{ By Inequality (\ref{4})})\\
\nonumber   &=& r \|f\|_{\widetilde{\omega}_1}.
\end{eqnarray}
Since $r<1$ is arbitrary, we have $\|f\|_{\omega op} \geq \|f\|_{\widetilde{\omega}_1}$.

Conversely, assume that $\omega$ does not satisfy F-property. So there exists a finite set $\{t_1,...,t_n\} \subset S$ and $0< r <1$ such that there is no $s \in S$ such that $\frac{\omega(t_i s)}{\omega(s)} \geq r \widetilde{\omega}(t_i) \; \; (1 \leq i \leq n)$.
Take $g = \sum\limits_{i=1}^{n} \delta_{t_i}$ and $E_i = \{ s \in S : \frac{\omega(t_i s)}{\omega(s)} \geq r \widetilde{\omega}(t_i) \}$ for $1 \leq i \leq n$. By definition of $\widetilde{\omega}$, each set $E_{i}$ is non-empty. On the other hand, by the assumption, $\mathop{\cap}\limits_{i=1}^{n} E_i = \phi$. Let $\alpha = \max \{ \sum\limits_{i=1}^{n} \alpha_{i} \widetilde{\omega}(t_i) : \alpha_i \in \{ r, 1 \} \text{ and }\alpha_i = r \text{ for some } i, \; 1 \leq i \leq n \}$. Then $0 < \alpha < \sum\limits_{i=1}^{n} \widetilde{\omega}(t_i)$. Let $f = \sum\limits_{s \in S} f(s) \delta_s \in \ell^1(S, \omega)$ such that $\| f \|_{\omega} \leq 1$. Then
\begin{eqnarray*}
  \|g \ast f \|_{\omega} &=& \| \sum\limits_{i=1}^{n} \delta_{t_i} \ast f \|_{\omega}  \\
   & \leq & \sum\limits_{i=1}^{n} \|\delta_{t_i} \ast f \|_{\omega} \\
   & = & \sum\limits_{i=1}^{n} \| \sum\limits_{s \in S} f(s) \delta_{t_i s} \|_{\omega} \\
   & \leq & \sum\limits_{i=1}^{n}  \sum\limits_{s \in S} |f(s)| \| \delta_{t_i s} \|_{\omega} \\
   &=& \sum\limits_{i=1}^{n}  \sum\limits_{s \in S} |f(s)| \omega(s) \frac{\omega(t_i s)}{\omega(s)}\\
   & \leq & \sum\limits_{s \in S} |f(s)| \omega(s) \alpha \quad ( \because \mathop{\cap}\limits_{i=1}^{n} E_i = \phi)\\
   & \leq & \alpha \quad ( \because \| f \|_{\omega} \leq 1)
\end{eqnarray*}
Thus $\| g \|_{\omega op} \leq \alpha < \| g \|_{\widetilde{\omega}}$.

$(iv)$ It is straightforward from $(\ref{9})$ above.

$(v)$ The weight $\omega(s) = 1 \; (s \in S)$ satisfies the hypothesis of $(\ref{23})$ above.
\end{proof}

\noindent Following example exhibits that Theorem \ref{17}($\ref{9}$) is not true if $\omega$ does not satisfy F-property.
\begin{exam}
Let $\mathbb{N}_{\wedge} = \mathbb{N}$ with binary operation $m \wedge n = \min \{m, n\} \; (m, n \in \mathbb{N})$. Define

$\omega(n) =
\left\{
	\begin{array}{ll}
		1  & \mbox{if } n=2 \\
		2 & \mbox{if } n=4 \\
        4^n  & \mbox{if } otherwise.
	\end{array}
\right.$\\
Then $\omega$ does not satisfy F-property and $\|\delta_1 + \delta_3 \|_{\omega op} \neq \| \delta_1 + \delta_3 \|_{\widetilde{\omega}_1}$.
\end{exam}

\begin{sol}
It is clear that $\widetilde{\omega}_1(1)=4$ and $\widetilde{\omega}_1(3)=32$. Now let $r = \frac{1}{2}$. Then $r \widetilde{\omega}_1(1)=2$ and $r\widetilde{\omega}_1(3)=16$. Now
\begin{eqnarray*}
\frac{\omega(1 \wedge 2)}{\omega(2)}&=&4 > r \widetilde{\omega}_1(1) > 2 \geq \frac{\omega(1 \wedge n)}{\omega(n)} \; \; (n \neq 2) \; \; \text{and}\\
\frac{\omega(3 \wedge 4)}{\omega(4)}&=&8 > r \widetilde{\omega}_1(3) > 1 \geq \frac{\omega(3 \wedge n)}{\omega(n)} \; \; (n \neq 4).
 \end{eqnarray*}
Hence $\omega$ does not satisfy F-property. Now we claim that $\|\delta_1 + \delta_3 \|_{\omega op} < \| \delta_1 + \delta_3 \|_{\widetilde{\omega}_1}$. Let $f \in \ell^1(\mathbb{N}_{\wedge}, \omega)$ such that $\|f\|_{\omega} = |f(2)| + 2 |f(4)| + \sum\limits_{n \neq 2,4}^{\infty} 4^n |f(n)| \leq 1$.
\begin{eqnarray*}
  \|(\delta_1 + \delta_3) \ast f\|_{\omega} &=& |f(1) + \sum\limits_{n=1}^{\infty} f(n)|\omega(1) + |f(2)|\omega(2) + |\sum\limits_{n \geq 3} f(n)| \omega(3)  \\
   & \leq & \| f \|_{\omega} + \sum\limits_{n=1}^{\infty} |f(n)| \omega(1) + \sum\limits_{n=3}^{\infty}|f(n)|\omega(3) \\
   & \leq & \| f \|_{\omega} + |f(2)|\omega(1) + |f(4)|\omega(1) + \sum\limits_{n \neq 2,4}|f(n)|\omega(1) + \sum\limits_{n=3}^{\infty}|f(n)| \omega(3) \\
   & \leq & \| f \|_{\omega} + 4 |f(2)|\omega(2) + 4|f(4)|\omega(4)+ \sum\limits_{n \neq 2,4}|f(n)|\omega(1) + \sum\limits_{n=3}^{\infty}|f(n)| \omega(3) \\
   &=& \| f \|_{\omega} + 4 \| f \|_{\omega} + 4 \| f \|_{\omega} + \| f \|_{\omega} + \| f \|_{\omega} \leq 11 \quad ( \because \| f \|_{\omega} \leq 1)
\end{eqnarray*}
Since $\| f \|_{\omega} \leq 1$ is arbitrary, $\|\delta_1 + \delta_3\|_{\omega op} \leq 11$ and $\| \delta_1 + \delta_3 \|_{\widetilde{\omega}_1}=36$.
\end{sol}

A semigroup $S$ is an \emph{ordered semigroup} if there is a partial order $\leq$ on $S$ such that, for any $s, t \in S$ with $s \leq t$, we have $us \leq ut$ and $su \leq tu$ for all $u\in S$. The partial order $\leq$ on $S$ is a \emph{total order} if, further, for each pair $s, t \in S$, either $s \leq t$ or $t \leq s$ holds \cite[Definition.1.2.11]{Dal:00}.

\begin{cor}  \label{20} Let $\omega$ be a weight on $S$. Then
 \begin{enumerate}[(i)]
  \item If $S$ is totally ordered and the map $\eta_t : S \longrightarrow (0, \infty)$ defined as $\eta_t(s) = \frac{\omega(ts)}{\omega(s)}$ is increasing (respectively, decreasing) for each $t \in S$, then $\omega$ has F-property.
   \item Let $S$ be a dense subsemigroup of $\bRp^{\bullet}$ and let $\omega$ be a weight on $S$ such that $\omega(s) \leq \limsup_{t \rightarrow 0^+} \frac{\omega(s+t)}{\omega(t)} \; (s \in S)$ in the usual topology. Then $\widetilde{\omega}_1 = \omega$, i.e., $\|\cdot\|_{\omega}$ is a regular norm on $\ell^1(S, \omega)$. \label{21}
 \end{enumerate}
\end{cor}

\begin{proof}
$(i)$ Assume that $\eta_t$ is increasing for each $t \in S$. Let $t_1, \ldots, t_n \in S$ and $0<r<1$. By the definition of $\widetilde{\omega}_1$, there exists $s_k \in S$ such that
\begin{equation}
  \eta_{t_k}(s_k) = \frac{\omega(t_k s_k)}{\omega(s_k)} \geq r \widetilde{\omega}_1(t_k) \quad (1 \leq k \leq n). \label{8}
\end{equation}
Set $s = \max \{s_1, \ldots, s_n\} \in S$. Then $s_k \leq s$ for each $k$. Since $\eta_{t_k}$ is increasing, we have $\eta_{t_k}(s_k) \leq \eta_{t_k}(s)$ for all $k$. Hence, by Inequality (\ref{8}), we have $$\frac{\omega(t_k s)}{\omega(s)} = \eta_{t_k}(s) \geq \eta_{t_k}(s_k) \geq r \widetilde{\omega}_1(t_k) \quad (1 \leq k \leq n).$$
Thus the weight $\omega$ has F-property.\\
If $\eta_t$ is decreasing, then take $s = \min \{ s_1, \ldots , s_n\}$ in the above proof.

  $(ii)$ Let $s \in S$. Note that $\widetilde{\omega}_1(s) \leq \omega(s)$ is always true. On the other hand,
$$\omega(s) \leq \limsup\limits_{t \rightarrow 0^+} \frac{\omega(s+t)}{\omega(t)} \leq \sup\limits_{t \in S} \frac{\omega(s+t)}{\omega(t)} = \widetilde{\omega}_1(s).$$ Hence $\widetilde{\omega}_1(s) = \omega(s)$. Thus $\|f\|_{\omega op} = \|f\|_{\omega} \; (f \in \ell^1(S, \omega))$. So $\| \cdot \|_{\omega}$ is regular.
\end{proof}

Finally, we note that the story for the Banach algebras $\ell_p(X, \omega)$ with pointwise product is totally different. First we define this Banach algebra. Let $X$ be any non-empty set. Let $\omega : X \longrightarrow [1, \infty)$ be any map. Let $1 \leq p < \infty$. Then the Banach space $\ell_p(X, \omega) = \{ f : X \longrightarrow \mathbb{C} : \| f \|_{p \omega} = \|f \omega\|_p < \infty\}$ is a commutative Banach algebra with respect to the pointwise product and the weighted norm $\| \cdot \|_{p \omega}$. It follows from the next result that the $\| \cdot \|_{p \omega}$ on $\ell^p$ is never regular.

\begin{thm}
Let $\ell_p(X, \omega)$ be as above. Then $\| f \|_{p \omega op} = \| f \|_{\infty} \; (f \in \ell_p(X, \omega))$. In particular, if $X$ is an infinite set, then $\| \cdot \|_{p \omega}$ is never regular.
\end{thm}

\begin{proof}
Let $f \in c_{00}(X)$ and $g \in \ell_p(X, \omega)$ with $\| g \|_{p \omega} \leq 1$. Then $\| f g \|_{p \omega} \leq \| f \|_{\infty}.$ Hence $\| f \|_{p \omega op} \leq  \| f \|_{\infty}$. For the reverse inequality, choose $x \in X$ such that $\| f \|_{\infty} = |f(x)|$. Let $\widetilde{\delta}_{x} = \frac{1}{\omega(x)} \delta_{x}$. Then $\|\widetilde{\delta}_{x}\|_{p \omega} = 1$ and $ \| f \|_{p \omega op} \geq \|f \cdot \widetilde{\delta}_{x}\|_{p \omega} = |f(x)| =\|f\|_{\infty}$. Thus $\| f \|_{p \omega op} = \|f\|_{\infty}$. By Lemma \ref{13}, $\| f \|_{p \omega op} = \|f\|_{\infty} \; (f \in \ell_p(X, \omega))$ because $c_{00}(X)$ is dense in $\ell_p(X, \omega)$ and $\| \cdot \|_{\infty} \leq \| \cdot\|_{p \omega}$. Let $X$ be an infinite set. Define $f_n = \sum\limits_{k=1}^{n} \frac{\delta_k}{\omega(k)}$. Then $\|f_n\|_{p \omega op} = \|f_n\|_{\infty} =1$ and   $\|f_n\|_{p \omega} =n$ for each $n \in \mathbb{N}$. Thus $\| \cdot \|_{p \omega}$ is not regular.
\end{proof}

\section{Examples of weights}
Now we give examples of weights having different properties. This should help us to understand the behaviour of the operator norm.\\
(1) Let $ \omega(s) = e^{-s^2} \; (s \in \bQp^{\bullet})$. By Corollary \ref{20}($\ref{21}$), $\|\cdot\|_{\omega}$ is a regular norm on $\ell^1(\bQp^{\bullet}, \omega)$. The same weight on $\mathbb{N}$ gives different result as in next Example.\\
(2) Let $\omega(n)= e^{-n^2} \; (n \in \mathbb{N})$. Then $\widetilde{\omega}_k(n) = e^{-n^2-2kn} \; (n \in \mathbb{N})$ and each $\widetilde{\omega}_k$ satisfies F-property. Hence, by Theorem \ref{17}($\ref{9}$), $\|f\|_{\widetilde{\omega}_k op} = \|f\|_{\widetilde{\omega}_{k+1}} \; (f \in \ell^1(\mathbb{N}, \omega))$. It is clear that $\ldots \lnapprox \| \cdot \|_{\widetilde{\omega}_{k+1}} \lnapprox \| \cdot \|_{\widetilde{\omega}_{k}}  \lnapprox \ldots \lnapprox \| \cdot \|_{\widetilde{\omega}_1} \lnapprox \| \cdot \|_{\omega}$. In particular, the norm $\| \cdot \|_{\omega}$ is not regular.\label{10}\\
(3) Let $\bQp^{\bullet}$ be the set of strictly positive rational numbers. Define $\omega(\frac{m}{n}) = n$ for $\frac{m}{n} \in \bQp^{\bullet}$ with $(m, n) = 1$, i.e., $m$ and $n$ are relatively prime. Then $\omega$ is a weight on $\bQp^{\bullet}$. Also $(pm+n, p) = 1$ for any prime number $p >n$. So that $\omega( \frac{m}{n} + \frac{1}{p}) = \omega(\frac{pm+n}{np}) = np$. Hence $\widetilde{\omega}_1(\frac{m}{n}) \geq \lim\limits_{ p \rightarrow \infty} \frac{\omega(\frac{m}{n} + \frac{1}{p})}{\omega(\frac{1}{p})} = n$. Thus $\widetilde{\omega}_1 = \omega$ and so $\| \cdot \|_{\omega}$ is a regular norm on $\ell^1(\bQp^{\bullet}, \omega)$.\\
(4) Let $\mathbb{N}_\wedge = \mathbb{N}$ with the binary operation $m \wedge n = \min \{m, n\}$. Then the norm $\| \cdot \|_1$ on $\ell^1(\mathbb{N}_\wedge)$ is regular. Infact, let $f = \sum\limits_{n=1}^{k} f(n) \delta_n \in c_{00}(\mathbb{N}_\wedge)$. Then $\|f\|_{1 op} \geq \| f \ast \delta_{k+1}\|_{1} = \| f \|_1$. But $\| f \|_{1 op} \leq \| f \|_1$ is always true. Now apply Lemma \ref{13}.\\
(5) Let $\mathbb{N}_l = \mathbb{N}$ with the binary operation $m \cdot n = m \; (m, n \in  \mathbb{N})$. Then $\mathbb{N}_l$ is not right cancellative. However, the norm $\| \cdot \|_1$ is regular on $\ell^1(\mathbb{N})$. Indeed, let $f \in \ell^1(\mathbb{N})$. Then $\|f\|_{1 op} \geq \| f \ast \delta_1  \|_1 = \|f\|_1$. But $\| f \|_{1 op} \leq \| f \|_1$ is always true.\\

\noindent
\textbf{Acknowledgement :} {The second author is thankful to the University Grants Commission (UGC), New Delhi, for providing Junior Research Fellowship.}

\end{document}